\documentclass[a4paper,11pt]{amsart}
\usepackage{fancyhdr}
\usepackage{soul}
\usepackage[T1]{fontenc}
\usepackage[english]{babel}
\usepackage[utf8]{inputenc}
\usepackage{amssymb,amsthm,amsmath} 
\usepackage{mathabx}
\usepackage{indentfirst}
\usepackage{color}
\usepackage{braket}
\usepackage[all,arc]{xy}
\usepackage{mathtools}
\usepackage{graphicx}
\usepackage{xcolor}
\usepackage[left=4cm, right=4cm, top=2cm]{geometry}
\usepackage{tikz-cd}
\usepackage[colorlinks]{hyperref}
\usepackage[shadow, color=babyblueeyes!]{todonotes}
\usepackage{booktabs}
\usepackage{xfrac} 
\usepackage{cite}
\usepackage[super]{nth}
\usepackage{cancel}
\usepackage{verbatim}
\usepackage{microtype}
\hypersetup
{
	colorlinks,
	linkcolor={black!50!black},
	citecolor={black!50!black},
	urlcolor={blue!80!black}
}
\definecolor{babyblueeyes}{rgb}{0.63, 0.79, 0.95}
\newtheorem{theorem}{Theorem}[section]

\newtheorem{proposition}[theorem]{Proposition}

\theoremstyle{definition}
\newtheorem{definition}[theorem]{Definition}

\theoremstyle{remark}

\newtheorem{remark}[theorem]{Remark}

\newcommand{\numberset}{\mathbb}

\newcommand{\K}{\numberset{K}}
\newcommand{\Z}{\numberset{Z}}
\newcommand{\C}{\numberset{C}}

\newcommand{\cat}[1]{\mathsf{#1}}

\newcommand{\Abtfold}{\K[x_{2L},x_{2R}]_{\{x_{2L}x_{2R}\}}}
\newcommand{\PPP}{e_{1L}e_{1R}}

\title{A very short note on the (rational) graded Hori map}

\author{Mattia Coloma}
\address{Università degli Studi di Roma "Tor Vergata"; Dipartimento di Matematica, Via della Ricerca Scientifica, 1 - 00133 - Roma, Italy; 
}
\email{coloma@mat.uniroma2.it}
\author{Domenico Fiorenza}
\address{Sapienza Universit\`a di Roma; Dipartimento di Matematica ``Guido Castelnuovo'', P.le Aldo Moro, 5 - 00185 - Roma, Italy; 
}
\email{fiorenza@mat.uniroma1.it}
\author{Eugenio Landi}
\address{Università di Roma Tre; Dipartimento di Matematica e Fisica Largo San Leonardo Murialdo, 1 - 00146 - Roma, Italy;
}
\email{eugenio.landi@uniroma3.it}
\begin{document}

\begin{abstract}
The graded Hori map has been recently introduced by Han-Mathai in the context of T-duality as a $\mathbb{Z}$-graded transform whose homogeneous components are the Hori-Fourier transforms in twisted cohomology associated with integral multiples of a basic pair of T-dual closed 3-forms. We show how in the rational homotopy theory approximation of T-duality, such a map is naturally realised as a pull-iso-push transform, where the isomorphism part corresponds to the canonical equivalence between the left and the right gerbes associated with a T-duality configuration. 
\end{abstract}
	\maketitle

\tableofcontents

\section{Foreword}
The graded Hori map has been recently introduced in \cite{Han-Mathai}, by assembling together the $\mathbb{Z}$-family of Hori maps associated with a certain $\mathbb{Z}$-family of T-duality configuration data naturally associated to a single T-duality configuration. This may at first sight appear as a rather ad hoc construction. The aim of this note is to show how, on the contrary, the graded Hori map as a whole naturally emerges from the geometry associated with a T-duality configuration. One only needs to look at a step higher with respect to the T-dual bundles: the graded Hori map is a manifestation of a canonical equivalence between the left and the right gerbes associated with a T-duality configuration. More precisely, we show that, 
in the rational homotopy theory approximation of T-duality, such a map is naturally realised as a pull-iso-push transform, where the isomorphism part corresponds to the left gerbe/right gerbe canonical equivalence.
\par
We will construct this pull-iso-push transform using only purely algebraic constructions related to the category $\cat{DGCA}$ of differential graded commutative algebras (DGCAs) over a characteristic zero field $\mathbb{K}$, which can be assumed to be the field $\mathbb{Q}$ of rational numbers. In particular, we will heavily use the language of extensions of DGCAs associated with DGCA cocycles. The reader familiar with rational homotopy theory will immediately recognise every step in the construction we are going to present as a translation of phenomena appearing in the rational homotopy theory approximation of T-duality. We point the unfamiliar reader to \cite{fiorenza2017tduality} for an introduction very close to the spirit of this note. We also borrow from \cite{fiorenza2016tduality-cocycles} the rational homotopy theory description of the equivalence of the gerbes associated to a T-duality configuration. See \cite{bunke-schick} for the topological origin of this equivalence.  
Here we choose to present the construction in purely algebraic terms, leaving to the reader the job of connecting to rational homotopy theory.
\par
The note is organised as follows. First we review topological $T$-duality in rational homotopy theory, in particular, in Section 2 we recall a few basic constructions on extensions of DGCAs and define the DGCA (co)classifying rational $T$-duality configurations, and in Section 3 we recall the definition of the two isomorphic rational gerbes associated with a rational $T$-duality configuration, whose isomorphism will be the ``iso'' part in the ``pull-iso-push'' transform.

After this review, in Section 4 we define the graded Hori map $\mathcal{T}_{L\to R}$ associated with these data and extend it to Laurent series. In Section 5 we show how, when the base field is the field $\mathbb{C}$ of complex numbers, this allows one to describe the graded Hori map as an operator on rings of meromorphic functions with a single pole at the origin taking values in a DGCA $A_0$ endowed with a rational T-duality configuration. It turns out that in this translation the graded Hori map becomes the antidiagonal matrix 
\[
\left(\begin{matrix}
 0 & 1 \\
-q\frac{d}{dq} & 0  \end{matrix}\right) 
\]
where $q$ is the complex coordinate on $\mathbb{C}$. Finally, in Section 6 we show how one can further extend coefficients to $A_0$-valued index $0$ Jacobi forms in the two variables $(z,\tau)\in \mathbb{C}\times \mathbb{H}$, by means of their $q$-expansion, where $q=e^{2\pi i z}$. This way we recover the original definition of the graded Hori map by Han-Mathai, as well as its main properties. In particular, one identifies the graded Hori map on Jacobi forms with the antidiagonal matrix 
\[
\left(\begin{matrix}
 0 & 1 \\
-\frac{1}{2\pi i}\frac{\partial}{\partial z} & 0  \end{matrix}\right),
\]
and therefore the composition of two graded Hori transforms as the operator $-\frac{1}{2\pi i}\frac{\partial}{\partial z}$ on the ring of $A_0$-valued index $0$ Jacobi-forms \cite[Theorem 2.2]{Han-Mathai}.

\section{Cocycles and extensions of DGCAs}
We start with a (non-negatively graded) differential graded commutative algebra $(A,d)$ over the field $\K$ and with a $2$-cocycle, i.e.,  a closed homogeneous element of degree $2$, $t_2\in A$. We can extend our base DGCA $A$ in such a way to trivialise the $2$-cocycle $t_2$ by adding a formal generator $e_1$ of degree $1$ and declaring our extension to be
\[
\begin{tikzcd}
{(A,d)} \arrow[r, "\iota", hook] & A_{\{t_2\}}:={(A[e_1], de_1 = t_2), }
\end{tikzcd}
\]
where the differential of $A_{\{t_2\}}$ coincides with the differential $d$ on the subalgebra $A$.
\par
The choice of a $2$-cocycle for the DGCA $A$ is the same datum as a DGCA map from the polynomial DGCA $(\K[x_2],0)$ to $A$, where $(\K[x_2],0)$ is the polynomial algebra over $\K$ on a single degree 2 generator $x_2$, endowed with the trivial differential. This in turn means regarding $A$ as an object under $(\K[x_2],0)$, a point of view that will be useful later.
\par
More generally, the datum of a DGCA map from $(\K[x_n],0)$ to $A$, where now $x_n$ is a degree $n$ variable, is the same as that of an $n$-cocycle in $A$ and, again, given such a cocycle $t_n \in A$ it is possible to extend $A$ to trivialise $t_n$ by \[
\begin{tikzcd}
{(A,d)} \arrow[r, "\iota", hook] & A_{\{t_n\}}:={(A[e_{n-1}], de_{n-1} := t_n) }.
\end{tikzcd}
\]
\par
The construction of $A_{\{t_n\}}$ out of the pair $(A,t_n)$ is universal: $A_{\{t_n\}}$ together with the embedding of the sub-DGCA $A$ is
the homotopy cofibre of $t_n$, i.e., the homotopy pushout of the diagram:

\[
\begin{tikzcd}
\K[x_{n}]\arrow[d]\arrow[r,"\psi_{t_n}"] & \left(A,d\right)\\
0 &
\end{tikzcd}
\]
of DGCAs, where $\psi_{t_n}$ is the unique DGCA morphism with $\psi(x_n)=t_n$, in the projective model structure on non-negatively graded DGCAs, see, e.g.,  \cite{bousfield-gugenheim}.
Indeed, in order to compute a model for this cofibre one has to replace the vertical map by a cofibration followed by a weak equivalence, and the easiest way of doing this is to consider 
\[
\K[x_{n}]\hookrightarrow \left(\K[x_{n},e_{n-1}], de_{n-1}=x_{n}\right)\cong 0,
\]
and then compute the ordinary pushout of the diagram
\[
\begin{tikzcd}
\K[x_{n}]\arrow[d]\arrow[r,"\psi_{t_n}"] & \left(A,d\right)\\
\left(\K[x_{n},e_{n-1}], de_{n-1}=x_{n}\right)
\end{tikzcd}
\]
to obtain 
\[
\begin{tikzcd}
\K[x_{n}]\arrow[d]\arrow[r,"\psi_{t_n}"] & \left(A,d\right) \arrow[d]\\
\left(\K[x_{n},e_{n-1}], de_{n-1}=x_{n}\right)\arrow[r] & (P,d_P),
\end{tikzcd}
\]
with
\begin{align*}
    \left(P,d_P\right)&=\left(A[e_{n-1}],d_Pa=da\text{ for }a\in A, d_Pe_{n-1}=\psi_{t_n}(x_{n})\right)\\
    &=\left(A[e_{n-1}], de_{n-1}=t_{n}\right)=A_{\{t_n\}}.
    \end{align*}

Universality implies in particular that the construction $(A,t_n)\rightsquigarrow A_{\{t_n\}}$ is natural, a fact that can also be easily checked directly: if $f\colon (A,t_n)\to (B,s_n)$ is a morphism of DGCAs endowed with $n$-cocycles, i.e., if $f$ is a morphism of DGCAs, $f\colon A\to B$, such that $f(t_n)=s_n$, then we get a morphism of DGCAs $\hat{f}\colon A_{\{t_n\}}\to B_{\{s_n\}}$ by setting $\hat{f}(a)=f(a)$ for any $a\in A$ and $\hat{f}(e_{n-1; A})=e_{n-1;B}$. This is manifestly compatible with compositions of morphisms of DGCAs endowed with $n$-cocycles.

\begin{remark}\label{rem:projection}
If $n$ is even, every degree $k$ element $a_k$ in $A_{\{t_n\}}$ can be uniquely written as $a_k=\alpha_k+e_{n-1} \beta_{k-n+1}$, for some degree $k$ element $\alpha_k$ and some degree $k-n+1$ element $\beta_{k-n+1}$ in $A$. The map 
\begin{align*}
    \pi\colon A_{\{t_n\}} &\to A[-n+1]\\
    \alpha_k+e_{n-1} \beta_{k-n+1}&\mapsto \beta_{k-n+1}
\end{align*}
is a map of chain complexes. Namely, we have
\begin{align*}
d_{[-n+1]}(\pi(a_k)&=d_{[-n+1]}(\pi(\alpha_k+e_{n-1} \beta_{k-n+1}))\\
&=d_{[-n+1]}\beta_{k-n+1}\\
&=(-1)^{(n-1)}d\beta_{k-n+1}
\end{align*}
and
\begin{align*}
\pi(d a_k)&=\pi(d(\alpha_k+e_{n-1} \beta_{k-n+1}))\\
&=\pi(d\alpha_k +t_n\beta_{k-n+1}
+(-1)^{n-1}d\beta_{k-n+1})\\
&=(-1)^{n-1}d\beta_{k-n+1}.
\end{align*}
Of course, $\pi$ is not a map of DGCAs (the shifted complex $A[-n+1]$ does not even have a natural DGCA structure). But it is a map of right DG-$A$-modules: if $\gamma_l$ is a degree $l$ element in $A$, then 
\[
\pi(a_k\gamma_l)=\pi((\alpha_k+t_n \beta_{k-n})\gamma_l)=\pi((\alpha_k\gamma_l)+t_n(\beta_{k-n}\gamma_l)=\beta_{k-n}\gamma_l=\pi(a_k)\gamma_l.
\]
As a side remark, by thinking of $\iota\colon A\to A_{\{t_n\}}$ as a pullback $p^*$ and of $\pi\colon A_{\{t_n\}} \to A[-n+1]$ as the pushforward $p_*$, the above identity is the projection formula:
\[
p_*(a_k p^*(\gamma_l))=p_*(a_k)\gamma_l.
\]
Finally, the map of right DG-$A$-modules $\pi\colon A_{\{t_n\}} \to A[-n+1]$ has an evident section 
\[
e_{n-1}\cdot -\colon A[-n+1] \to A_{\{t_n\}}
\]
given by the left multiplication by $e_{n-1}$.
\end{remark}

\medskip 

An example of the  construction $(A,t_n)\rightsquigarrow A_{\{t_n\}}$ we will be interested in is the following. Consider the polynomial algebra 
\[
\K[x_{2L},x_{2R}]\cong \K[x_{2L}]\otimes \K[x_{2R}]
\]
on two degree 2 generators $x_{2L}$ and $x_{2R}$, endowed with the trivial differential.\footnote{Here and below, all tensor products are over $\mathbb{K}$.} Then the element $x_{2L}x_{2R}$ is a 4-cocycle and so defines a DGCA map
\begin{align*}
\K[t_4]&\to \K[x_{2L},x_{2R}]\\
t_4&\mapsto x_{2L}x_{2R}
\end{align*}
The associated extension is the DGCA 
\[
\K[x_{2L},x_{2R}]_{\{x_{2L}x_{2R}\}}=(\K[x_{2L},x_{2R},y_3], dx_{2L}=dx_{2R}=0, dy_3=x_{2L}x_{2R})
\]
Notice that $\Abtfold$ carries two distinguished 2-cocycles $x_{2L}$ and $x_{2R}$ and that
$\sigma\colon x_{2L}\leftrightarrow x_{2R}$ is a DGCA automorphism of $\Abtfold$ exchanging the two cocycles. We denote by $p_L,p_R\colon \K[x_2]\to \Abtfold$ the two maps corresponding to the cocycles $x_{2L},x_{2R}$, respectively.

\section{Two equivalent rational gerbes}
In order to get the DGCA construction corresponding to the rational homotopy description of the pull-iso-push transform between gerbes associated with a T-duality configuration, we consider a DGCA $A$ together with a map $\Abtfold \overset{f}{\to} A$. As we noticed above, the source of $f$ has two distinct $2$-cocycles corresponding to maps $p_L, p_R: \K[x_2] \to \Abtfold$ sending the generator $x_2$ in $x_{2L}$ and in $x_{2R}$, respectively. Composing with the map $f$ we therefore get maps $f_L,f_R: \K[x_2] \to A$, corresponding to two distinct $2$-cocycles in $A$, and we end up with following commutative diagram of DGCAs:
\[
\begin{tikzcd}[row sep=large, column sep=small]
  &                                                          & {\K[x_2]} 
  \arrow[ ldd, "p_L"'] \arrow[rdd, "p_R"] \arrow[bend right=20,lld, "f_L"'] \arrow[bend left=20,rrd, "f_R"] &                                      &   \\
A &                                                          &                                                                                         &                                      & A \\
  & \Abtfold \arrow[rr, "\sigma",leftrightarrow] \arrow[bend left=10,lu, "f"] &                                                                                         & \Abtfold \arrow[bend right=10, ru, "f"'] &  
\end{tikzcd}
\]
The previous diagram shows that the map $f$ can be read in two different ways as a map in the undercategory ${}^{\K[x_2]/}\cat{DGCA}$ of DGCAs endowed with a distinguished $2$-cocycle, i.e., with a distinguished morphism from $\K[x_2]$. In particular, we have that $f$ is both a map between $\Abtfold$ and $A$ decorated with their left $2$-cocycles
\[
\begin{tikzcd}
{(\Abtfold,p_L)} \arrow[r, "f"] & {(A,f_L)},
\end{tikzcd}
\]
and with their right $2$-cocycles
\[
\begin{tikzcd}
{(\Abtfold,p_R)} \arrow[r, "f"] & {(A,f_R)}.
\end{tikzcd}
\]
This will be crucial in order to define the equivalence between the algebraic structures corresponding to the left and right gerbes of topological T-duality. We begin with the following, which is a particular case of the ``hofib/cyc adjunction'' of \cite{fiorenza2016tduality-cocycles,fiorenza2017tduality}, and whose proof in this specific case we give for the sake of completeness.
\begin{proposition}\label{1.2}
Let $(A,t_2)$ be a DGCA with a distinguished $2$-cocycle $t_2$. Then the association
\[
\begin{tikzcd}[column sep=0.3cm]
\mathrm{Hom}^{}_{\small{{}^{\K[x_2]/}\cat{DGCA}}}\left((\Abtfold,p_L),(A,\psi_{t_2})\right) \arrow[r]& \mathrm{Hom}_{\small{\cat{DGCA}}}\left(\K[x_3],A_{\{t_2\}}\right)\\
\varphi\arrow[r, maps to]&\tilde{\varphi}
\end{tikzcd}
\]
where $\tilde{\varphi}$ is defined by
\[
\tilde{\varphi}\colon x_3\mapsto\varphi(y_3)-e_1\varphi(x_{2R}),
\]
is a natural bijection. Clearly, everything identically works if we exchange $p_L$ with $p_R$ and $x_{2R}$ with $x_{2L}$.
\end{proposition}
\begin{proof}
We begin by showing that $\tilde{\varphi}(x_3)$ is a $3$-cocycle. If 
\[
\varphi\colon (\Abtfold,p_L)\to (A,\psi_{t_2})
\]
is a map in the undercategory ${}^{\K[x_2]/}\cat{DGCA}$ , then
\[
\varphi(x_{2L})=(\varphi\circ p_L)(x_2)=\psi_{t_2}(x_2)=t_2.
\]
Therefore,
\begin{align*}
d(\tilde{\varphi}(x_3))&=
d(\varphi(y_3)-e_1\varphi(x_{2R}))=\\
&=\varphi(x_{2L})\varphi(x_{2R})-t_2\varphi(x_{2R})+e_1\varphi(dx_{2R})=\\
& = 0.
\end{align*}
This shows that the map $\varphi\mapsto \tilde{\varphi}$ actually takes values in $\mathrm{Hom}_{\small{\cat{DGCA}}}(\K[x_3],A_{\{t_2\}})$. Now we define a map in the opposite direction. For a DGCA morphism $\psi\colon \K[x_3]\to A_{\{t_2\}}$, let $t_3$ be the $3$-cocycle $t_3=\psi(x_3)$ in $A_{\{t_2\}}$. The $3$-cocycle $t_3$ can be uniquely written as $t_3 = a_3-e_1b_2$ with $a_3,b_2\in A$. The association
\[
y_3 \mapsto a_3, \ \ \ x_{2R} \mapsto b_2, \ \ \ x_{2L} \mapsto t_2 
\]
defines a map $\tilde{\psi}\colon (\Abtfold,p_L)\to (A,t_2)$ in ${}^{\K[x_2]/}\cat{DGCA}$. It is immediate to check that $\tilde{\tilde{\varphi}}=\varphi$ and $\tilde{\tilde{\psi}}=\psi$, so the two maps are inverse each other.
\end{proof}

Now, let us go back to our DGCA $A$ endowed with a DGCA morphism $\Abtfold \overset{f}{\to} A$.
To avoid confusion, let us denote by $e_{1L}$ and $e_{1R}$ the additional degree 1 generators in  the extensions $A_L:=A_{f(x_{2L})}$ and  $A_R:=A_{f(x_{2R})}$ of $A$, respectively.
By the above proposition, and looking at $f$ both as a morphism from $(\Abtfold,p_L)$ to $(A,f_L)$ and as a morphism from $(\Abtfold,p_R)$ to $(A,f_R)$, we end up with distinguished $3$-cocycles
\[
f(y_3)-e_{1L}f(x_{2R}) \in A_L,\ \ \ f(y_3)-e_{1R}f(x_{2L}) \in A_R
\]
and again, we can define extensions of $A_L$ and $A_R$ by trivialising the above $3$-cocycles. We define the \emph{left rational gerbe} $\mathcal{G}_L$ and the \emph{right rational gerbe} $\mathcal{G}_R$ of the rational $T$-configuration $f$ 
as the DGCAs
\begin{align*}
\mathcal{G}_L &:= {{A_L}_{\{f(y_3)-e_{1L}f(x_{2R})\}}}\\
\mathcal{G}_R &:= {{A_R}_{\{f(y_3)-e_{1R}f(x_{2L})\}}}.
\end{align*}
Again, to avoid confusion, we denote by $\xi_{2L}$ and $\xi_{2R}$ the additional degree $2$ generators of $\mathcal{G}_L$ and $\mathcal{G}_R$ as extensions of $A_L$ and of $A_R$, respectively.
Both $\mathcal{G}_L$ and $\mathcal{G}_R$ are extensions of $A$ (since both $A_L$ and $A_R$ were extensions), and this tower of extensions of $A$ can be depicted in the diagram
\[
\begin{tikzcd}
\mathcal{G}_L &                                              &                                        &                                              & \mathcal{G}_R \\
              & A_L  \arrow[lu, "i_L"] &                                                & A_R  \arrow[ru, "i_R"'] &               \\
              &                                              & A \arrow[lu, "\iota_L"] \arrow[ru, "\iota_R"'] &                                              &              
\end{tikzcd}
\]
We can add to this diagram the DGCA $A_{LR}:= A_L \otimes_A A_R$, i.e., the DGCA $(A[e_{1L},e_{1R}], de_{1L}=f(x_{2L}),de_{1R}=f(x_{2R}))$, obtaining the diagram
\[
\begin{tikzcd}
\mathcal{G}_L &                                              & A_{LR}                                         &                                              & \mathcal{G}_R \\
              & A_L \arrow[ru, "\iota_R"'] \arrow[lu, "i_L"] &                                                & A_R \arrow[lu, "\iota_L"] \arrow[ru, "i_R"'] &               \\
              &                                              & A \arrow[lu, "\iota_L"] \arrow[ru, "\iota_R"'] &                                              &              
\end{tikzcd}
\]
where the central square commutes. As a matter of notation, in the above diagram we are writing $\iota_L$ (resp. $\iota_R$) wherever the extension is made by means of the $1$-form $e_{1L}$ (resp. $e_{1R}$) and $i_L$ (resp. $i_R$) whenever the extension is made by means of the $2$-form $\xi_{2L}$ (resp. $\xi_{2R}$).

We can extend $\mathcal{G}_L$ and $\mathcal{G}_R$ by computing the obvious (homotopy) pushouts to get the further extensions

\[\begin{tikzcd}
                                    & {\mathcal{G}_L}_{\{ f(x_{2R})\}}                     &                                                & {\mathcal{G}_R}_{\{ f(x_{2L})\}}                     &                                      \\
\mathcal{G}_L \arrow[ru, "\iota_R"'] &                                              & A_{LR} \arrow[ru, "i_R"'] \arrow[lu, "i_L"]    &                                              & \mathcal{G}_R \arrow[lu, "\iota_L"] \\
                                    & A_L \arrow[ru, "\iota_R"'] \arrow[lu, "i_L"] &                                                & A_R \arrow[lu, "\iota_L"] \arrow[ru, "i_R"'] &                                      \\
                                    &                                              & A \arrow[lu, "\iota_L"] \arrow[ru, "\iota_R"'] &                                              &                                     
\end{tikzcd}
\]
Explicitly,
\[
{\mathcal{G}_L}_{\{ f(x_{2R})\}} =
\left(A[e_{1L},e_{1R},\xi_{2L}], \begin{cases}
de_{1L}=f(x_{2L})\\
de_{1R}=f(x_{2R})\\
d\xi_{2L}=f(y_3)-e_{1L}f(x_{2R})
\end{cases}\right)
\]
and
\[
{\mathcal{G}_R}_{\{ f(x_{2L})\}} =
\left(A[e_{1L},e_{1R},\xi_{2R}], \begin{cases}
de_{1L}=f(x_{2L})\\
de_{1R}=f(x_{2R})\\
d\xi_{2R}=f(y_3)-e_{1R}f(x_{2L})
\end{cases}\right).
\]

We can now make explicit the iso part of our pull-iso-push transform.
\begin{proposition}
The DGCAs ${\mathcal{G}_L}_{\{ f(x_{2R})\}}$ and ${\mathcal{G}_R}_{\{ f(x_{2L})\}}$ are isomorphic via an isomorphism
\[
\begin{tikzcd}
{\mathcal{G}_L}_{\{ f(x_{2R})\}} \arrow[r, "\nu"] & {\mathcal{G}_R}_{\{ f(x_{2L})\}}
\end{tikzcd}
\]
that is the identity on $A_{LR}$ and acts as 
\[
\xi_{2L} \mapsto \xi_{2R}+\PPP,
\]
on the degree two generator. The inverse isomorphism is, clearly, $\nu^{-1}\colon \xi_{2R} \mapsto \xi_{2L}-\PPP$.
\end{proposition}
\begin{proof}
The map $\nu$ is a map of graded commutative algebras, and it is of course a bijection since an explicit inverse is given by the map of graded commutative algebras $\nu^{-1}$ which is the identity on $A_{LR}$ and sending $\xi_{2R}$ to $\xi_{2L}-e_{1L}e_{1R}$. To see that $\nu$ is a map of DGCAs we need to show that it is a map of chain complexes. This can be checked on the generators of the polynomial algebra ${\mathcal{G}_L}_{\{ f(x_{2R})\}}$, so we only need to compute $d\nu(\xi_{2L})$. We have
\begin{align*}
d\nu(\xi_{2L})&=
d(\xi_{2R}+e_{1L}e_{1R})=\\ 
&= f(y_3)-e_{1R}f(x_{2L}) +f(x_{2L})e_{1R}-e_{1L}f(x_{2R})=\\
&= f(y_3)-e_{1L}f(x_{2R})=\\
&= \nu(f(y_3)-e_{1L}f(x_{2R}))=\\
&= \nu(d\xi_{2L}), 
\end{align*}
where we used that $f(y_3)-e_{1L}f(x_{2R})\in A_{LR}$ and $\nu$ is the identity on $A_{LR}$.
\end{proof}
The isomorphisms $\nu$ and $\nu^{-1}$ complete our previous diagram to the commutative diagram
\[
\begin{tikzcd}
                                     & {\mathcal{G}_L}_{\{ f(x_{2R})\}} \arrow[bend left=10,rr, "\nu"]   &                                               & {\mathcal{G}_R}_{\{ f(x_{2L})\}}\arrow[bend left=10,ll, "\nu^{-1}"]                     &                                     \\
\mathcal{G}_L \arrow[ru, "\iota_R"'] &                                              & A_{LR} \arrow[lu, "i_L"] \arrow[ru, "i_R"']   &                                              & \mathcal{G}_R \arrow[lu, "\iota_L"] \\
                                     & A_L \arrow[ru, "\iota_R"'] \arrow[lu, "i_L"] &                                               & A_R \arrow[lu, "\iota_L"] \arrow[ru, "i_R"'] &                                     \\
                                     &                                              & A \arrow[lu, "\iota_L"] \arrow[ru, "\iota_R"'] &                                              &                                    
\end{tikzcd}
\]
\section{The graded Hori map from rational equivalences of gerbes}

All the maps and the DGCAs appearing in the upper part of our diagram
\[
\begin{tikzcd}
                                     & {\mathcal{G}_L}_{\{ f(x_{2R})\}} \arrow[rr, "\nu"] &  & {\mathcal{G}_R}_{\{ f(x_{2L})\}} &                                     \\
\mathcal{G}_L \arrow[ru, "\iota_R"'] &                                            &  &                          & \mathcal{G}_R \arrow[lu, "\iota_L"]
\end{tikzcd}
\]
can be extended to the rings of (bounded above) formal Laurent series in the degree $2$ generators.
For instance, as a graded commutative algebra the DGCA $\mathcal{G}_L$ is the polynomial algebra $A_L[\xi_{2L}]$ over $A_L$ and so embeds as a subalgebra into the ring of Laurent series
\[
\widehat{\mathcal{G}_L}:=A_L[[\xi_{2L}^{-1}]][\xi_{2L}]=: A_L[[\xi_{2L}^{-1},\xi_{2L}].
\]
The ring $\widehat{\mathcal{G}_L}$ has moreover a natural DGCA structure, by setting 
\[
d\xi_{2L}^{-1}=-\xi_{2L}^{-2}\left(f(y_3)-e_{1L}f(x_{2R})\right),
\]
making ${\mathcal{G}_L}\hookrightarrow \widehat{\mathcal{G}_L}$ an inclusion of DGCAs. One similarly extends the other DGCAs $\mathcal{G}_R, {\mathcal{G}_L}_{\{ f(x_{2R})\}}$ and ${\mathcal{G}_R}_{\{ f(x_{2L})\}}$ appearing in the above diagram.

The maps $\iota_R, \iota_L$ obviously extend to the rings of Laurent series. We denote by $\hat{\iota}_L$ and $\hat{\iota}_R$ these extensions. We notice that $\nu$ extends too, we only need to be careful in defining the extension $\hat{\nu}$. As $\PPP$ is nilpotent, this is done by using the formal power series inverse for $1-\eta$, i.e., by declaring that the action of $\hat{\nu}$ on $\xi_{2L}^{-1}$ is given by
\[
\hat{\nu}(\xi_{2L}^{-1})= (\xi_{2R}+\PPP)^{-1} = \sum_{i \geq 0}(-1)^i(\PPP)^i\xi_{2R}^{-i-1} = \xi_{2R}^{-1}-\PPP\xi_{2R}^{-2},
\]
where we used that $(\PPP)^2=0$. One easily checks that $\hat{\nu}$ is indeed a DGCA morphism: it is compatible with the relation $\xi_{2L}^{-1}\xi_{2L}=1$ as
\[
\hat{\nu}(\xi_{2L}^{-1})\hat{\nu}(\xi_{2L})=(\xi_{2R}^{-1}-\PPP\xi_{2R}^{-2})
(\xi_{2R}+\PPP)=1
\]
and with the differential as
\begin{align*} 
&\hat{\nu}(d\xi_{2L}^{-1})=\hat{\nu}(-\xi_{2L}^{-2}\left(f(y_3)-e_{1L}f(x_{2R})\right))\\
&=-(\xi_{2R}^{-1}-\PPP\xi_{2R}^{-2})^2(f(y_3)-e_{1L}f(x_{2R}))\\
&=-(\xi_{2R}^{-2}-2\PPP\xi_{2R}^{-3})(f(y_3)-e_{1L}f(x_{2R}))\\
&=-\xi_{2R}^{-2}f(y_3) +2\PPP\xi_{2R}^{-3} f(y_3)+\xi_{2R}^{-2}e_{1L}f(x_{2R})
\end{align*}
and
{\small\begin{align*}
&d\hat{\nu}(\xi_{2L}^{-1})=d(\xi_{2R}^{-1}-\PPP\xi_{2R}^{-2})
\\
&=-\xi_{2R}^{-2}\left(f(y_3)-e_{1R}f(x_{2L})\right)-(d\PPP)\xi_{2R}^{-2}+
2\PPP\xi_{2R}^{-3}\left(f(y_3)-e_{1R}f(x_{2L})\right)\\
&=-\xi_{2R}^{-2}\left(f(y_3)-e_{1R}f(x_{2L})\right)-(f(x_{2L})e_{1R}-e_{1L}f(x_{2R})
)\xi_{2R}^{-2}+
2\PPP\xi_{2R}^{-3}f(y_3)\\
&=-\xi_{2R}^{-2}f(y_3)+e_{1L}f(x_{2R})
\xi_{2R}^{-2}+
2\PPP\xi_{2R}^{-3}f(y_3)
\end{align*}}

As $f(x_{2L})$ is an even cocycle, by Remark \ref{rem:projection} we have a projection
\[
\pi\colon {\mathcal{G}_R}_{\{ f(x_{2L})\}} \to {\mathcal{G}_R}[-1]
\]
mapping $\alpha_k+e_{1L} \beta_{k-1}$ to $\beta_{k-1}$, which is a morphism of right DG-${\mathcal{G}_R}$-modules. 
Also the projection $\pi\colon \colon {\mathcal{G}_R}_{\{ f(x_{2L})\}} \to {\mathcal{G}_R}[-1]$ naturally extends to formal Laurent series modules to a map
\[
\hat{\pi}\colon \widehat{{\mathcal{G}_R}_{\{ f(x_{2L})\}}} \to \widehat{{\mathcal{G}_R}}[-1]
\]
and so it is possible to build a pull-iso-push transform $\mathcal{T}_{L \to R}$ as the composition
\[
\begin{tikzcd}
                                                                             & \widehat{{\mathcal{G}_L}_{\{ f(x_{2R})\}}} \arrow[rr, "\hat{\nu}"] &  & \widehat{{\mathcal{G}_R}_{\{ f(x_{2L})\}}} \arrow[rd, "\hat{\pi}"'] &                     \\
\widehat{\mathcal{G}_L} \arrow[ru, "\hat{\iota}_R"'] \arrow[rrrr, "\mathcal{T}_{L \to R}"] &                                                            &  &                                                          & \widehat{\mathcal{G}_R}[-1].
\end{tikzcd}
\]

\medskip

The transform 
\[
\mathcal{T}_{L\to R}\colon \widehat{\mathcal{G}_L} \to \widehat{\mathcal{G}_R}[-1]
\]
associated to the initial rational $T$-duality configuration $\Abtfold \overset{f}{\to} A$ is seen to coincide with the graded Hori map introduced by  Han and Mathai in \cite{Han-Mathai}. Namely, the action of $\hat{\nu}$ on a generic degree $k$ element 
\[
\omega_k=\sum_{n\in \mathbb{Z}}(\alpha_{2n+k}+e_{1L}\beta_{2n+k-1}+e_{1R}\gamma_{2n+k-1}+\PPP\delta_{2n+k-2})\xi_{2L}^{-n}
\]
in $\widehat{{\mathcal{G}_L}_{\{ f(x_{2R})\}}}$ is given by

{\footnotesize\begin{align*}
&\hat{\nu}(\omega_k)=
\sum_{n\in \mathbb{Z}}(\alpha_{2n+k}+e_{1L}\beta_{2n+k-1}+e_{1R}\gamma_{2n+k-1}+\PPP\delta_{2n+k-2})\hat{\nu}(\xi_{2L}^{-n})=\\ &=
\sum_{n\in \mathbb{Z}}(\alpha_{2n+k}+e_{1L}\beta_{2n+k-1}+e_{1R}\gamma_{2n+k-1}+\PPP\delta_{2n+k-2})(\xi_{2R}^{-n}-n\PPP\xi_{2R}^{-n-1}) =\\
&= 
\sum_{n\in \mathbb{Z}}(\alpha_{2n+k}+e_{1L}\beta_{2n+k-1}+e_{1R}\gamma_{2n+k-1}+\PPP\delta_{2n+k-2})\xi_{2R}^{-n}-n\PPP\alpha_{2n+k}\xi_{2R}^{-n-1}=\\
&= 
\sum_{n\in \mathbb{Z}}(\alpha_{2n+k}+e_{1L}\beta_{2n+k-1}+e_{1R}\gamma_{2n+k-1}+\PPP(\delta_{2n+k-2}-(n-1)a_{2n+k-2}))\xi_{2R}^{-n},
\end{align*}}

hence the action of $\hat{\nu}$ on the coefficients of a generic degree $k$ Laurent series in $\widehat{{\mathcal{G}_L}_{\{ f(x_{2R})\}}}$ is given by
\[
\begin{pmatrix}\alpha_{2n+k} \\ \beta_{2n+k-1} \\ \gamma_{2n+k-1} \\ \delta_{2n+k-2}\end{pmatrix} \overset{\hat{\nu}}{\mapsto} \begin{pmatrix}\alpha_{2n+k} \\ \beta_{2n+k-1} \\ \gamma_{2n+k-1} \\ \delta_{2n+k-2}-(n-1)\alpha_{2n+k-2}\end{pmatrix}
\]
The inclusion $\widehat{\mathcal{G}_L} \overset{\hat{\iota_R}}{\to} \widehat{{\mathcal{G}_L}_{\{ f(x_{2R})\}}}$ and the projection $\hat{\pi}\colon \widehat{{\mathcal{G}_R}_{\{ f(x_{2L})\}}}\to \widehat{{\mathcal{G}_R}}[-1]$
can be displayed in a similar way 
\[
\begin{pmatrix}\alpha_{2n+k} \\ b_{2n+k-1}\end{pmatrix} \overset{\hat{\iota_R}}{\mapsto} \begin{pmatrix}\alpha_{2n+k} \\ \beta_{2n+k-1} \\ 0 \\ 0\end{pmatrix}
\qquad;\qquad 
\begin{pmatrix}\alpha_{2n+k} \\ \beta_{2n+k-1} \\ \gamma_{2n+k-1} \\ \delta_{2n+k-2}\end{pmatrix} \overset{\hat{\pi}}{\mapsto} \begin{pmatrix}\beta_{2n+k-1} \\ \delta_{2n+k-2}\end{pmatrix}    
\]
The left-to-right transform $\mathcal{T}_{L\to R}$ therefore act on the coefficients of a generic degree $k$ element $\sum_{n \in \Z}(\alpha_{2n+k}+e_{1L}\beta_{2n+k-1})\xi_{2L}^{-n}\in 
\widehat{\mathcal{G}_L}$ as
\[
\begin{pmatrix}\alpha_{2n+k} \\ \beta_{2n+k-1}\end{pmatrix} \overset{\hat{\iota_R}}{\mapsto} \begin{pmatrix}\alpha_{2n+k} \\ \beta_{2n+k-1} \\ 0 \\ 0\end{pmatrix}\overset{\hat{\nu}}{\mapsto} \begin{pmatrix}\alpha_{2n+k} \\ \beta_{2n+k-1} \\ 0 \\ -(n-1)\alpha_{2n+k-2}\end{pmatrix} \overset{\hat{\pi}}{\mapsto} \begin{pmatrix}\beta_{2n+k-1} \\ -(n-1)\alpha_{2n+k-2}\end{pmatrix}
\]
i.e., it acts on the degree $k$ element $\sum_{n\in \mathbb{Z}}(\alpha_{2n+k}+e_{1L}b_{2n+k-1})\xi_{2L}^{-n} \in \widehat{\mathcal{G}_L}$ 
as
\begin{align*}
   \sum_{n\in \mathbb{Z}}(\alpha_{2n+k}+e_{1L}\beta_{2n+k-1})\xi_{2L}^{-n} &\mapsto \sum_{n \in \Z}\left(\beta_{2n+k-1}-(n-1)e_{1R}\alpha_{2n+k-2}\right)\xi_{2R}^{-n}\\
   &= \sum_{n \in \Z}\beta_{2n+k-1}\xi_{2R}^{-n}+e_{1R}\sum_{n \in \Z}-n\alpha_{2n+k}\xi_{2R}^{-n-1}
\end{align*}
The above expressions can be conveniently packaged by introducing, for every sequence $\{\eta_{2n+k}\}_{n\in \mathbb{Z}}$ of elements of $A$ with $\deg(\eta_{2n+k})=2n+k$, the Laurent series in a degree $2$ variable $\xi$
\[
\eta_{(k)}(\xi)=\sum_{n\in \mathbb{Z}}\eta_{2n+k}\xi^{-n}.
\]
We have manifest isomorphisms of graded vector spaces
\[
\begin{tikzcd}
\widehat{{\mathcal{G}_L}_{\{ f(x_{2R})\}}} & \left(\begin{matrix}A[[\xi^{-1},\xi]\\
\oplus\\ A[[\xi^{-1},\xi][-1]\\ \oplus \\ A[[\xi^{-1},\xi][-1]\\ \oplus \\ A[[\xi^{-1},\xi][-2] \end{matrix}\right)
\arrow[l, "\sim"']\arrow[r, "\sim"] & \widehat{{\mathcal{G}_R}_{\{ f(x_{2L})\}}}
\end{tikzcd}
\]
and
\[
\begin{tikzcd}
\widehat{{\mathcal{G}_L}} & \left(\begin{matrix}A[[\xi^{-1},\xi] \\ \oplus \\
A[[\xi^{-1},\xi][-1]\end{matrix}\right)
\arrow[l, "\sim"']\arrow[r, "\sim"] & \widehat{{\mathcal{G}_R}}.
\end{tikzcd}
\]
In terms of these isomorphisms, the maps $\hat{\nu}$, $\hat{\iota}_R$ and $\hat{\pi}$ are represented by the following matrices:
\[
\hat{\nu}\mapsto 
\left(
\begin{matrix}
1 & 0 & 0 &0\\
0 & 1 & 0 &0\\
0 & 0 & 1 &0\\
\frac{d}{dz} & 0 & 0 &1\\
\end{matrix}
\right); \qquad 
\hat{\iota}_R\mapsto 
\left(
\begin{matrix}
1 & 0 \\
0 & 1\\
0 & 0\\
0 & 0
\end{matrix}
\right);\qquad 
\hat{\pi}\mapsto 
\left(
\begin{matrix}
0 & 1 & 0 &0\\
0 & 0 & 0 &1
\end{matrix}
\right),
\]
so that the graded left-to-right Hori transform $\mathcal{T}_{L\to R}$ is represented in matrix form as: 
\[
\mathcal{T}_{L\to R}\mapsto 
\left(
\begin{matrix}
0 & 1 \\
\frac{d}{d\xi} & 0 
\end{matrix}
\right).
\]
One similarly defines the right-to-left Hori transform $\mathcal{T}_{R\to L}$. As 
\[
\left(
\begin{matrix}
0 & 1 \\
\frac{d}{d\xi} & 0 
\end{matrix}
\right)\left(
\begin{matrix}
0 & 1 \\
\frac{d}{d\xi} & 0 
\end{matrix}
\right)=
\left(
\begin{matrix}
\frac{d}{d\xi} & 0 \\
0 &\frac{d}{d\xi}  
\end{matrix}
\right)
\]
one sees that
\[
\mathcal{T}_{R\to L}\circ \mathcal{T}_{L\to R} = \frac{d}{d\xi_{2L}}\colon \widehat{{\mathcal{G}_L}}\to \widehat{{\mathcal{G}_L}}[-2]
\]
and 
\[
\mathcal{T}_{L\to R}\circ \mathcal{T}_{R\to L} = \frac{d}{d\xi_{2R}}\colon \widehat{{\mathcal{G}_R}}\to \widehat{{\mathcal{G}_R}}[-2].
\]

\section{Hori transforms of meromorphic functions}\label{sec:meromorphic} Before extending the ring of coefficients to the ring of Jacobi forms we start with a one variable intermediate step. We will need an extra degree $2$ variable in order to keep the following computations within the context of graded maps. So we assume that our base DGCA $A$ is of the form
\[
A:=A_0[u^{-1},u]]
\]
where $u$ is a degree $2$ variable and $A_0$ is a DGCA endowed with a rational T-duality configuration $\Abtfold \overset{f}{\to} A_0$. Notice the a T-duality configuration on $A_0$ induces a T-duality configuration 
\[
\Abtfold \overset{f}{\to} A_0 \hookrightarrow A
\]
on $A$ simply by composing $f$ with the inclusion $A_0 \hookrightarrow A$. All the extension and gerbes below are computed with respect to this T-duality configuration on $A$. For instance, the extended left gerbe $\widehat{\mathcal{G}_L}$ will be
\[
\widehat{\mathcal{G}_L} =\hskip -20pt \LaTeXoverbrace{A_0}^{\text {starting DGCA}} \hskip -15pt \LaTeXunderbrace{[u^{-1},u]]}_{\vbox{\hsize 1.2 cm \scriptsize \noindent additional variable}} \hskip -20pt\LaTeXoverbrace{[e_{1L}]}^{\text{trivialise $f(x_{2L})$}} \hskip -27 pt\LaTeXunderbrace{[[\xi_{2L}^{-1},\xi_{2L}]}_{\vbox{\hsize 2.3 cm \scriptsize \noindent trivialise  $f(y_3)-e_{1L}f(x_{2R})$ and extend to Laurent series}}\hskip -10pt.
\]

Assume now the base field $\K$ to be the field $\C$ of complex numbers and let  $\mathcal{M}_0$ be the $\C$-algebra of meromorphic functions on $\C$ that are holomorphic on the punctured plane $\C\setminus\{0\}$, i.e. meromorphic functions that admit at most a polar singularity in the origin. By looking at the algebra $\mathcal{M}_0$ as a DGCA concentrated in degree zero, we can then consider the DGCA $\mathcal{M}_0(A_0):= \mathcal{M}_0\otimes A_0$, that we will call the DGCA of meromorphic functions with values in $A_0$ and with at most polar singularities in the origin. A degree $k$ element in $\mathcal{M}_0(A_0)$ has a Laurent series expansion around the origin of the form
\[
f(q) = \sum_{n}f_{n;k}q^{n}
\]
where the $f_{n;k}$ are degree $k$ elements in $A_0$, with $f_{n;k} = 0$ for $n \ll 0$. For any $i\in \Z$ we have an isomorphism $\mu_{2i}$ of graded vector spaces
\begin{align*} 
\mathcal{M}_0(A_0) \xrightarrow{\mu_{2i}} & {A[[\xi^{-1},\xi][2i]}\\
f(q) &\mapsto \xi^{i}f(u\xi^{-1})
\end{align*}
Notice that there exists a commutative diagram
\[
\begin{tikzcd}
\mathcal{M}_0(A_0) \arrow[r, "-q\frac{d}{dq}"] \arrow[d, "\mu_0"'] & \mathcal{M}_0(A_0) \arrow[d, "\mu_{-2}"] \\
{A[[\xi^{-1},\xi]} \arrow[r, "\frac{d}{d\xi}"]                    & {A[[\xi^{-1},\xi][-2]}                  
\end{tikzcd}
\]
As remarked at the end of the previous section, the natural isomorphism of graded vector spaces of $\widehat{\mathcal{G}_L}$ and $\widehat{\mathcal{G}_R}$ with $A[[\xi^{-1},\xi] \oplus A[[\xi^{-1},\xi][-1]$ identifies the graded Hori map $\mathcal{T}_{L \to R}$ with the antidiagonal matrix $\left(\begin{smallmatrix}
 0 & 1 \\
\frac{d}{d\xi} & 0  \end{smallmatrix}\right)$, i.e., we have a commutative diagram
\[
\begin{tikzcd}[ampersand replacement=\&]
A[[\xi^{-1},\xi]\oplus A[[\xi^{-1},\xi][-1] \arrow[rr, "{\left(\begin{smallmatrix}
 0 & 1 \\
\frac{d}{d\xi} & 0  \end{smallmatrix}\right)}"]    \arrow[d, "\wr"]                        \&\& A[[\xi^{-1},\xi][-1]\oplus A[[\xi^{-1},\xi][-2]   \arrow[d, "\wr"]   \\
\widehat{\mathcal{G}_L} \arrow[rr, "\mathcal{T}_{L\to R}"]                           \&\& \widehat{\mathcal{G}_R}[-1]                                     
\end{tikzcd}
\]
Therefore, we see that the graded Hori map $\mathcal{T}_{L \to R}$ participates into a commutative diagram of graded vector spaces
\[
\begin{tikzcd}[ampersand replacement=\&]
\mathcal{M}_0(A_0)\oplus \mathcal{M}_0(A_0)[-1] \arrow[rr, "{\left(\begin{smallmatrix}
 0 & 1 \\
-q\frac{d}{dq} & 0  \end{smallmatrix}\right)}"] \arrow[d, "{\mu_0\oplus \mu_0[-1]}"', "\wr"] \&                                       \& \mathcal{M}_0(A_0)[-1]\oplus \mathcal{M}_0(A_0) \arrow[d, "\wr"', "{\mu_0[-1]\oplus \mu_{-2}}"] \\
A[[\xi^{-1},\xi]\oplus A[[\xi^{-1},\xi][-1] \arrow[rr, "{\left(\begin{smallmatrix}
 0 & 1 \\
\frac{d}{d\xi} & 0  \end{smallmatrix}\right)}"]    \arrow[d, "\wr"]                        \&\& A[[\xi^{-1},\xi][-1]\oplus A[[\xi^{-1},\xi][-2]   \arrow[d, "\wr"]   \\
\widehat{\mathcal{G}_L} \arrow[rr, "\mathcal{T}_{L\to R}"]                           \&\& \widehat{\mathcal{G}_R}[-1]                                     
\end{tikzcd}
\]
The same happens for the graded Hori map $\mathcal{T}_{R \to L}$, so that 
the composition $\mathcal{T}_{R \to L}\circ \mathcal{T}_{L \to R}$
is identified with the endomorphism
\[
\left(
\begin{matrix}
-q\frac{d}{dq} &0\\
0 & -q\frac{d}{dq}
\end{matrix} 
\right)
\]
of $\mathcal{M}_0(A_0)\oplus \mathcal{M}_0(A_0)[-1]$, and similarly for $\mathcal{T}_{L \to R}\circ \mathcal{T}_{R \to L}$.

\section{Extending coefficients to the ring of Jacobi forms}
In this concluding section we extend the ring of coefficients for our extended gerbes to the graded ring of Jacobi forms of index $0$. We address the reader to the classic \cite{eichler1985theory} for a complete and detailed account of the general theory of 
Jacobi forms of arbitrary index, and here we content us in briefly recalling the definition of a (meromorphic) Jacobi form of index 0.
\begin{definition}
A \emph{(meromorphic) Jacobi form of weight $s$ and index $0$} is a function
\[
\begin{tikzcd}
\C \times \mathbb{H} \arrow[r, "J"] & \C
\end{tikzcd}
\]
which is meromorphic in the variable $z$ and holomorphic in the variable $\tau$, such that $J$ 
\begin{itemize}
    \item is modular in $\tau$ i.e. $J(\frac{z}{c\tau+d},\frac{a\tau+b}{c\tau+d}) = (c\tau+d)^sJ(z,\tau)$ for any $\left(\begin{smallmatrix} a & b \\ c & d\end{smallmatrix}\right)$ in $SL(2,\Z)$;
    \item is elliptic in $z$ i.e. $J(z+\lambda \tau+\mu,\tau) =J(z,\tau)$ for any $(\lambda, \mu)$ in $\Z^2$;
    \item has a polar behaviour for $z\to +i\infty$.
\end{itemize}
\end{definition}

We notice two important features of Jacobi forms.
First, by applying the operator $\partial/\partial z$ to both sides of the modularity and of the ellipticity equations, one sees  that if $J(z,\tau)$ is a Jacobi form of weight $s$ and index $0$ then $\frac{\partial }{\partial z}J(z,\tau)$ is a Jacobi form of weight $s+1$ and index $0$.
\par
Secondly, from the ellipticity condition for the pair $(0,1)\in \Z^2$ one sees that every Jacobi form is periodic in $z$ of period $1$,
hence it has a series expansion in the variable $q= e^{2\pi i z}$ of the form
\[
J(z,\tau) = \sum_{n =n_0}^\infty\alpha_n(\tau)q^{n},
\]
for some $n_0\in \mathbb{Z}$, where the fact that this Laurent series is bounded below is a consequence of the polar behaviour of $J$ for $z\to +i\infty$.
\par
As the weight $s$ ranges over the integers, Jacobi form of index $0$
form a graded ring 
\[
\mathcal{J}_0 =\bigoplus_{s\in \Z}\mathcal{J}_0(s),
\]
(with degree given by the weight). The fact that $\frac{\partial }{\partial z}$ maps weight $s$ index $0$ Jacobi forms to weight $s+1$ index $0$ Jacobi forms then can be expressed by saying that $\frac{\partial }{\partial z}$ is a degree 1 derivation of the graded ring $\mathcal{J}_0$. Moreover, from the identity
\[
-q\dfrac{\partial}{\partial q} = -\dfrac{1}{2\pi i}\dfrac{\partial}{\partial z}
\]
we see that the ring of $q$-expansions of index $0$ Jacobi forms (a subring of the ring of bounded below Laurent series in the variable $q$ with coefficients in the ring of holomorphic function on $\mathbb{H}$) is closed under the action of the operator $-q\dfrac{\partial}{\partial q}$. 
\par
We can now verbatim repeat the construction of Section \ref{sec:meromorphic}. For a DGCA $B$ over $\mathbb{C}$, let us write $B(\tau)$ for the DGCA
\[
B(\tau):=B\otimes \mathcal{H}ol(\mathbb{H})
\]
where the ring $\mathcal{H}ol(\mathbb{H})$ of holomorphic function in the variable $\tau\in \mathbb{H}$ is seen as a DGCA concentrated in degree zero.
Also, let us write 
\[
\mathcal{J}_0(A_0) =\bigoplus_{s\in \Z}\mathcal{J}_0^{(s)}(A_0)
\]
for the bigraded ring $\mathcal{J}_0(A_0):=\mathcal{J}_0\otimes A_0$ of index $0$ Jacobi forms with values in a DGCA $A_0$. Then the commutative diagram at the end of Section \ref{sec:meromorphic} induces the commutative diagram
\[
\begin{tikzcd}[ampersand replacement=\&, column sep=small]
\mathcal{J}_0^{(s_1)}(A_0) \oplus \mathcal{J}_0^{(s_2)}(A_0)[-1] 
\arrow[rr, "{\left(\begin{smallmatrix}
 0 & 1 \\
-\frac{1}{2\pi i}\frac{\partial}{\partial z} & 0  \end{smallmatrix}\right)}"] 
\arrow[d, "q\text{-expansion}"']
\&\& \mathcal{J}_0^{(s_2)}(A_0)[-1] \oplus \mathcal{J}_0^{(s_1+1)}(A_0)
\arrow[d, "q\text{-expansion}"]
\\
\mathcal{M}_0(A_0(\tau))\oplus \mathcal{M}_0(A_0(\tau))[-1] \arrow[rr, "{\left(\begin{smallmatrix}
 0 & 1 \\
-q\frac{\partial}{\partial q} & 0  \end{smallmatrix}\right)}"] \arrow[d, "{\mu_0\oplus \mu_0[-1]}"', "\wr"] \&                                       \& \mathcal{M}_0(A_0(\tau))[-1]\oplus \mathcal{M}_0(A_0(\tau)) \arrow[d, "\wr"', "{\mu_0[-1]\oplus \mu_{-2}}"] \\
A(\tau)[[\xi^{-1},\xi]\oplus A(\tau)[[\xi^{-1},\xi][-1] \arrow[rr, "{\left(\begin{smallmatrix}
 0 & 1 \\
\frac{d}{d\xi} & 0  \end{smallmatrix}\right)}"]    \arrow[d, "\wr"]                        \&\& A(\tau)[[\xi^{-1},\xi][-1]\oplus A(\tau)[[\xi^{-1},\xi][-2]   \arrow[d, "\wr"']   \\
\widehat{\mathcal{G}_L}(\tau) \arrow[rr, "\mathcal{T}_{L\to R}"]                           \&\& \widehat{\mathcal{G}_R}(\tau)[-1]                                     
\end{tikzcd}
\]
That is, the graded Hori transform $\mathcal{T}_{L\to R}$ induces the
morphism
\[
\left(\begin{matrix}
 0 & 1 \\
-\frac{1}{2\pi i}\frac{\partial}{\partial z} & 0  \end{matrix}\right)
\colon 
\mathcal{J}_0^{(s_1)}(A_0) \oplus \mathcal{J}_0^{(s_2)}(A_0)[-1]
\to
\mathcal{J}_0^{(s_2)}(A_0)[-1] \oplus \mathcal{J}_0^{(s_1+1)}(A_0)
\]
at the level of index $0$ $A_0$-valued Jacobi forms, for any weights $s_1,s_2$ in $\mathbb{Z}$. The same holds for the graded Hori transform $\mathcal{T}_{R\to L}$, so that the composition $\mathcal{T}_{L\to R}\circ \mathcal{T}_{R\to L}$ acts as
{\small\[
\left(\begin{matrix}
-\frac{1}{2\pi i}\frac{\partial}{\partial z} & 0  \\
0 & -\frac{1}{2\pi i}\frac{\partial}{\partial z}\end{matrix}\right)
\colon 
\mathcal{J}_0^{(s_1)}(A_0) \oplus \mathcal{J}_0^{(s_2)}(A_0)[-1]
\to
\mathcal{J}_0^{(s_1+1)}(A_0) \oplus \mathcal{J}_0^{(s_2+1)}(A_0)[-1]
\]}
and similarly for $\mathcal{T}_{L \to R}\circ \mathcal{T}_{R \to L}$. This reproduces \cite[Theorem 2.2]{Han-Mathai}.

\subsection*{Acknowledgements} d.f thanks NYU-AD for support on occasion of the workshop {\it M-theory and Mathematics} during which the idea of this note originated, and Hisham Sati and Urs Schreiber for discussions and comments on an early version of this note.

\nocite{*}
\bibliographystyle{alpha}
\bibliography{bibliography}
\end{document}